\documentclass[twoside,11pt]{amsart}
\usepackage{amssymb, hyperref}

\newtheorem{thm}{Theorem}[section]

\newtheorem{prop}[thm]{Proposition}

\newtheorem{lem}[thm]{Lemma}
\newtheorem{cor}[thm]{Corollary}

\newtheorem{ex}[thm]{Example}

\newcommand{\mm}{\mathfrak m}
\newcommand{\m}{\mathfrak m}
\newcommand{\N}{\mathbb N}
\newcommand{\pd}{\mathop{\mathrm{pd}}\nolimits}

\theoremstyle{definition}

\theoremstyle{remark}
\newtheorem{rmk}[thm]{Remark}

\theoremstyle{definition}

\theoremstyle{remark}

\numberwithin{equation}{section}

\begin{document}

\title[A multiplicity bound and a criterion for the Cohen-Macaulay property]{A multiplicity bound for graded rings and a criterion for the Cohen-Macaulay property}


\author[Huneke]{Craig Huneke}
\address{Department of Mathematics, University of Virginia, Charlottesville, VA 22904.}
\email{huneke@virginia.edu}
\thanks{}

\author[Mantero]{Paolo Mantero}
\address{Department of Mathematics, University of California Riverside, Riverside, CA 92521.}
\email{mantero@math.ucr.edu\newline
\indent{\it URL:} \href{http://math.ucr.edu/~mantero/}{\tt http://math.ucr.edu/$\sim$mantero/}}

\author[McCullough]{Jason McCullough}
\address{Department of Mathematics, Rider University,  Lawrence Township, NJ 08648.}
\email{jmccullough@rider.edu}
\thanks{}

\author[Seceleanu]{Alexandra Seceleanu}
\address{Department of Mathematics, University of Nebraska at Lincoln, Lincoln, NE 68588.}
\email{aseceleanu2@math.unl.edu}
\thanks{}

\subjclass[2010]{Primary: 13C14; Secondary: 13H15, 13D40.}

\date{January 23, 2014}


\commby{}

\begin{abstract}
Let $R$ be a polynomial ring over a field. We prove an upper bound for the multiplicity of $R/I$ when $I$ is a homogeneous ideal of the form $I=J+(F)$, where $J$ is a Cohen-Macaulay ideal and $F\notin J$. The bound is given in terms of two invariants of $R/J$ and the degree of $F$.
We show that ideals achieving this upper bound have high depth, and provide a purely numerical criterion for the Cohen-Macaulay property. Applications to quasi-Gorenstein rings and almost complete intersections are given.
\end{abstract}

\maketitle

\section{introduction}
Bounds on the multiplicity of a ring $S$ in terms of other invariants of $S$ (and analyses of the rings achieving these bounds) have attracted a strong interest over the last 130 years, from the classical lower bound ${\rm deg}\,X\geq {\rm codim}\,X +1$ for non-degenerate projective varieties to the still open Eisenbud-Green-Harris Conjecture \cite{EGH}, or the Huneke-Srinivasan Multiplicity Conjecture, proved a few years ago by Eisenbud and Schreyer \cite{ES}.

In the present paper, we prove a new upper bound for the multiplicity of a wide class of graded rings and study the defining ideals achieving this bound. Let $R$ be a polynomial ring over a field $k$, $J$ a homogeneous Cohen-Macaulay ideal, $F\notin J$ a homogeneous element and $I=J+(F)$. Let $e(S)$ denote the multiplicity of a graded ring $S$. If $F$ is regular on $R/J$, it is well-known that $e(R/I)=e(R/J)\cdot{\rm deg}(F)$. If $F$ is a zero-divisor on $R/J$ (that is, ${\rm ht}\,I={\rm ht}\,J$) one has the elementary inequality $e(R/I)\leq e(R/J)-1$. In the present paper, we prove the following sharper upper bound
\begin{equation}\label{Ineq}
e(R/I)\leq e(R/J) - {\rm max}\{1,s-{\rm deg}(F) +1\},
\end{equation}
where $s=s(R/J)$ is the difference between the smallest graded shift appearing in the last step of a minimal graded free resolution of $R/J$ and ${\rm ht}\,J$. 
This bound, by its nature, is more restrictive when $\deg(F)$ is small or when $R/J$ is level, a particular instance of which is when $R/J$ is Gorenstein.

The inequality (\ref{Ineq}) generalizes a previous result of Engheta  bounding the multiplicity of a homogeneous almost complete intersection in terms of the degrees of its minimal generators \cite[Theorem~1]{En}, see Corollary \ref{En2}.

We name ideals $I$ achieving equality in (\ref{Ineq}) {\it ideals of maximal multiplicity}. Our second main result states that these ideals define factor rings with high depth. More precisely, $R/I$ is Cohen-Macaulay if the degree of $F$ is sufficiently small, namely, ${\rm deg}(F)\leq s$. In the remaining cases, that is, when ${\rm deg}(F)>s$, we show ${\rm depth}(R/I)= {\rm dim}(R/I)-1$, provided $R/J$ is Gorenstein (Theorem \ref{Main}). 
A consequence of this result is that almost complete intersections generated in a single degree and having maximal multiplicity are Cohen-Macaulay (Corollary \ref{onedeg}), a result that we employ in \cite{HMMS2}.
Also, one should contrast this result with the abundance of examples of almost complete intersection ideals $I$ (even generated in a single degree) of any other multiplicity that are not Cohen-Macaulay (see Remark \ref{4quad}). 

We then provide a sufficient condition for Cohen-Macaulay ideals to have maximal multiplicity, and exhibit examples, which include rational normal curves and $\m$-primary ideals (see Theorem \ref{char}, Examples \ref{m-prim} and \ref{aci} and Corollary \ref{matrix}). Among the applications, we use linkage to deduce a lower bound for the multiplicity of any graded quasi-Gorenstein ring $S$ in terms of its $a$-invariant and dimension (Proposition \ref{boundQ}). If this lower bound is achieved,  
then $S$ is actually Gorenstein (see Theorem \ref{Gor}). 
We also prove that ideals of maximal multiplicity are unmixed if and only if they are Cohen-Macaulay (Corollary \ref{Unm}). 

As a final application, we remark that part of the material from this paper is employed in the forthcoming paper \cite{HMMS2} where, motivated by a question of Stillman \cite{PeS}, we prove a close-to-optimal upper bound on the projective dimension of any ideal generated by 4 quadratic polynomials. One of our original motivations in the present paper was, in fact, to find a more structural  reason for the fact (proved in an earlier draft of \cite{HMMS2}) that any almost complete intersection ideal (not necessarily unmixed) of multiplicity $6$ generated by $4$ quadrics  is Cohen-Macaulay.\\
\\
The structure of the paper is the following: in Section 2, we prove the three main results, namely the upper bound (\ref{Ineq}), the high depth properties of ideals of maximal multiplicity and a sufficient condition for Cohen-Macaulay ideals to have maximal multiplicity. We employ these results to obtain a sufficient condition for almost complete intersections of quadrics to be Cohen-Macaulay and provide examples. In Section 3, we prove a lower bound for the multiplicity of quasi-Gorenstein rings, a multiplicity-based sufficient condition for quasi-Gorenstein rings to be Gorenstein, and analogies between ideals of maximal multiplicity and ideals of multiplicity one.

\section{The main results}

Throughout this paper $R$ is a polynomial ring over a field $k$ and $\m$ denotes its unique homogeneous maximal ideal. Since we may harmlessly replace $k$ by $k(X)$, by base change we may always assume $|k|=\infty$.\\
\\
Recall that  the {\it Hilbert function} of a finitely generated graded $R$-module $M$ is the numerical function given by $HF_{M}(i)={\rm dim}_kM_i$ for all $i\in \mathbb Z$. If $A$ is a graded artinian factor ring of $R$, then $e(A)=\sum_{i=0}^{\infty}HF_A(i)$, and the {\it socle} of $A$ is the $k$-vector space ${\rm Soc}(A)=0:_A\mm_A$, where $\mm_A=\m A$. A {\it socle element} is an element of ${\rm Soc}(A)$.

Let $J$ be a homogeneous Cohen-Macaulay $R$-ideal of height $g$. We consider the invariant
$$s(R/J)={\rm min}\{i\,|\, {\rm Tor}_g^R(R/J,k)_i\neq 0\}-g.$$
We remark here two additional interpretations of this invariant. First, when $|k|=\infty$, 
$s(R/J)$ is the smallest degree of a non-zero homogeneous socle element of a general artinian reduction of $R/J$
(a {\it general artinian reduction} of $T=R/J$ is a ring $T/(L_1,\ldots,L_d)$ where $L_i$ are general linear forms and $d={\rm dim}(T)$). Second, one has $s(R/J)=c(R/J)+{\rm dim}(R/J),$ where  $c(R/J)=-{\rm max}\{i\,|\, [k \otimes_{R/J} \omega_{R/J}]_i \neq 0\}$ and $\omega_{R/J}$ denotes the graded canonical module of $R/J$. 

We also recall that the {\it unmixed part} of an $R$-ideal $K$, denoted $K^{un}$, is the intersection of the primary components of $K$ of minimal height. The ideal $K$ is {\it unmixed} if $K=K^{un}$. The associativity formula 
yields $e(R/K)=e(R/K^{un})$ and the following easy remark:

\begin{rmk}\label{unm}
If $K\subseteq L$ are unmixed ideals of the same height, then $e(R/K)\geq e(R/L)$, and $e(R/K)=e(R/L)$ if and only if $K=L$.
\end{rmk}

We now prove our first main result. 

\begin{thm}\label{bound}
Let $J$ be a homogeneous Cohen-Macaulay $R$-ideal, and let $F\notin J$ be homogeneous. Set $I=J+(F)$ and assume ${\rm ht}\,I={\rm ht}\,J$. Then 
\begin{itemize}
\item[(i)]  $e(R/I)\leq e(R/J)-{\rm max}\{1, s(R/J)-{\rm deg}(F) +1\};$
\item[(ii)] if equality is achieved in $($i$)$ and ${\rm deg}(F)\leq s(R/J)$, then $R/I$ is Cohen-Macaulay.
\end{itemize}
\end{thm}

\begin{proof}
We prove both statements at once. Set $\delta={\rm deg}(F)$ and $s=s(R/J)$.
First, assume $1> s-\delta+1$. We need to show $e(R/I)\leq e(R/J)-1$. This follows by the equality $e(R/I)=e(R/I^{un})$, the inclusions $J\subsetneq I\subseteq I^{un}$ and Remark \ref{unm}. Hence, we may assume $1\leq s-\delta+1$, that is, $\delta\leq s$. We need to show that $e(R/I)\leq e(R/J)-(s-\delta+1)$ and, if equality is achieved, then $R/I$ is Cohen-Macaulay.

Let $A$ be a general artinian reduction of $R/J$, and denote by $^-$ images in $A$. We claim that both statements hold if $\overline{F}\neq 0$ in $A$. Indeed, in this case, we set $n={\rm max}\{t\in \N_0\,|\,\overline{F}\mm_A^t\neq 0\}<\,\infty$ and observe that  $HF_{A/(\overline{F})}(i)\leq HF_A(i)-1$ for every $\delta \leq i\leq \delta+n$. This proves $e(A/(\overline{F}))\leq e(A)-(n+1)$. Let $G$ be a homogeneous element of degree $n$ such that $\overline{F}\overline{G}\neq 0$ in $A$. By definition of $n$, we have $0\neq \overline{F}\overline{G}\in {\rm Soc}(A)$, whence $\delta + n ={\rm deg}(\overline{F}\overline{G})\geq s$. This proves $s-\delta\leq n$ and gives
$e(A/(\overline{F}))\leq e(A)-(n+1)\leq e(A)-(s-\delta +1)$. 

A well-known result of Serre \cite[Theorem~4.7.10]{BH} yields $e(A)=e(R/J)$ and $e(R/I)\leq e(A/(\overline{F}))$, hence we obtain the inequalities
$$e(R/I)\leq e(A/(\overline{F}))\leq e(A)-(s-\delta+1)= e(R/J)-(s-\delta+1).$$
If, moreover, $e(R/I)= e(R/J)-(s-\delta+1)$, then we have the equality $e(R/I)=e(A/(\overline{F}))$ and, by \cite[Theorem~4.7.10]{BH}, $R/I$ is Cohen-Macaulay.

Therefore, to finish the proof it suffices to prove $\overline{F}\neq 0$ in $A$. Set $\widetilde{R}=R/J$ and let ${}\,\widetilde{}\,{}$ denote images in $\widetilde{R}=R/J$. Since $A$ is a general artinian reduction of $\widetilde{R}$, to show $\overline{F}\neq 0$ in $A$, we need to prove that $\widetilde{F}$ is not in the intersection of all the (general) minimal reductions of $\widetilde{\mm}$, that is, $\widetilde{F}\notin {\rm core}(\widetilde{\mm})$ in $\widetilde{R}=R/J$, where ${\rm core}(\mm)$ denotes the core of $\mm$. Since $\widetilde{R}=R/J$ is Cohen-Macaulay, it follows by work of Fouli, Polini and Ulrich, \cite[Corollary~4.3.(b)]{FPU} that
${\rm core}(\widetilde{\mm})\subseteq \widetilde{\mm}^{s+1}.$
The inequality $s\geq\delta$ now yields $\widetilde{F}\notin {\rm core}(\widetilde{\mm})$ for degree reasons.
\end{proof}

As a first application we recover an upper bound (first proved by Engheta \cite{En}) for the multiplicity of $R/I$, where $I$ is any homogeneous almost complete intersection.

\begin{cor}$(${\rm Engheta,} \cite[Theorem~1]{En}$)$\label{En2}
Let $I=(f_1,\ldots,f_{g+1})$ be a homogeneous almost complete intersection of height $g$ in $R$, where $f_1,\ldots,f_{g}$ form a homogeneous regular sequence. Set $d_{i}={\rm deg}\,(f_i)$ for every $i=1,\ldots,g+1$. Then,
 $$e(R/I)\leq \prod_{i=1}^gd_i-{\rm max}\left\{1, \sum_{i=1}^g(d_i-1)-(d_{g+1}-1)  \right\}.$$
In particular, if $I$ is generated by forms of the same degree $d$ and $g>1$ then
 $$e(R/I)\leq d^g-(d-1)(g-1).$$
\end{cor} 

\begin{proof}
Since $J=(f_1,\ldots,f_g)$ is a complete intersection, we have $s(R/J)=\sum_{i=1}^g(d_i-1)$. Now, Theorem \ref{bound} applied to $I=J+(f_{g+1})$ proves the statement.
\end{proof}

Although the bound of Corollary \ref{En2} can be sharp (see Remark \ref{4quad}) G. Caviglia remarked that, if this is the case, then either $g=2$, or $g=3$ and  $d_1=d_2=d_3=2$. 
Instead, the bound of Theorem \ref{bound} is more general and, indeed, is achieved in a wider variety of situations (see, for instance, Examples \ref{m-prim} and \ref{aci} and Corollary \ref{matrix}).
\medskip

Before stating our second main result, we recall a few definitions. An ideal $I$ is called {\it almost Cohen-Macaulay} if ${\rm depth}(R/I)\geq {\rm dim}(R/I)-1$. Also, a homogeneous Cohen-Macaulay ideal $J$ of height $g$ is said to be {\it level} if there exists only one positive integer $i$ such that $[{\rm Tor}_g^R(R/J,k)]_i\neq 0$. An  ideal
$J$ is {\it Gorenstein} if $R/J$ is Gorenstein, that is, ${\rm Tor}_g^R(R/J,k)\cong k$. Clearly, every homogeneous Gorenstein ideal is level.

We will also need a few definitions and results from liaison theory. Two (homogeneous) ideals $J$ and $K$ in $R$ are {\it linked}, denoted $J\sim K$, if there exists a (homogeneous) Gorenstein ideal $G$ with $K=G:_RJ$ and $J=G:_RK$. Sometimes we say $J$ and $K$ are linked {\it by $G$}. Linkage (by complete intersections) has been studied since the nineteenth century, although its first modern treatment appeared in the ground-breaking paper by Peskine and Szpiro \cite{PS}. Properties of liaison by Gorenstein ideals were then studied in \cite{Sc3}. We refer the interested reader to \cite{M}, \cite{HU} and their references.

\begin{prop}$(${\rm Peskine-Szpiro, Schenzel} \cite{PS}, \cite{Sc3}$)$\label{PS}
Let $J$ be an unmixed ideal of $R$ of height $g$. If $G\subseteq J$ is a height $g$ Gorenstein ideal, and $K=G:J$, then $J\sim K$.
\end{prop}
We will use the following simple fact: If $G$ is a Gorenstein ideal contained in an ideal $J$ and ${\rm ht}\,G={\rm ht}\,J$, then $G:J=G:J^{un}$, that is, $G:J$ is linked to $J^{un}$. 

\begin{prop}$(${\rm Peskine-Szpiro, Golod, Schenzel} \cite{PS}, \cite{Go}, \cite{Sc3}$)$\label{linkage}
If $J\sim K$, then
\begin{itemize}
\item[$($a$)$] $S/J$ is Cohen-Macaulay if and only if $S/K$ is Cohen-Macaulay$;$
\item[$($b$)$] $e(S/J)+e(S/K)=e(S/G)$, where $G$ is the ideal defining the link $J\sim K$.
\end{itemize}
\end{prop}

The following linkage result is well-known (its proof follows, for instance, along the same lines of the proof of \cite[Theorem~3]{En2}).

\begin{lem}\label{ses}
Let $I=J+(F)$, where $J$ is Gorenstein ideal with ${\rm ht}(J)={\rm ht}(I)$, and $F\notin J$. If $L$ is any ideal linked to $I^{un}$, then ${\rm pd}(R/I)\leq {\rm pd}(R/L)+1.$
\end{lem}

In the setting of Theorem \ref{bound}, we say that $I$ has {\it maximal multiplicity} if there exist a Cohen-Macaulay ideal $J$ and an element $F\notin J$ with $I=J+(F)$, ${\rm ht}(I)={\rm ht}(J)$ and $e(R/I)=e(R/J)-{\rm max}\left\{1,s(R/J)-{\rm deg}(F)+1\right\}$. If $J$ and $F$ are as above, we say they form a {\it maximal decomposition} of $I$.

We can now state our second main result, proving the high depth of $R/I$.
\begin{thm}\label{Main}
Let $I$ be a homogeneous $R$-ideal of maximal multiplicity, and let $I=J+(F)$ be a maximal decomposition of $I$.
 \begin{itemize}
 \item[(a)] If ${\rm deg}(F)\leq s(R/J)$, then $R/I$ is Cohen-Macaulay.
 \item[(b)] If $R/J$ is level, then $R/I$ is Cohen-Macaulay if and only if ${\rm deg}(F)\leq s(R/J)$.
 \item[(c)] If ${\rm deg}(F)> s(R/J)$ and $R/J$ is Gorenstein, then ${\rm depth}(R/I)={\rm dim}(R/I)-1$. 
 \end{itemize}
\end{thm}
Note that, although $J$ is Cohen-Macaulay, the ideal $I$ may not even be unmixed.
\begin{proof} Set $g={\rm ht}\,J={\rm ht}\,I$.
Assertion (a) was proved in Theorem \ref{bound} (ii). To prove assertion (b) we need to show that if $R/I$ is Cohen-Macaulay, then ${\rm deg}(F)\leq s(R/J)$. Let $L_1,\ldots,L_d$ be general linear forms, where $d={\rm dim}(R/J)={\rm dim}(R/I)$, and let $^-$ denote images modulo $L_1,\ldots,L_d$. Since $R/I$ is Cohen-Macaulay, by Theorem \cite[Theorem~4.7.10]{BH}, we have $e(R/I)=e(\overline{R}/\overline{I})$. Moreover, since $R/J$ is level, we have ${\rm Soc}(\overline{R}/\overline{J})=\overline{\mm_{R/J}}^s$, where $s=s(R/J)$. Now, assume by contradiction ${\rm deg}(F)> s(R/J)$. Then $\overline{F}\in \overline{\mm_{R/J}}^{s+1}=0$ in $\overline{R}/\overline{J}$, that is, $\overline{R}/\overline{I}=\overline{R}/\overline{J}$. This implies 
$$e(R/I)=e(\overline{R}/\overline{I})=e(\overline{R}/\overline{J})=e(R/J).$$
Since both $I$ and $J$ are unmixed, Remark \ref{unm} implies $J=I$, which is a contradiction.

We now prove assertion (c). The assumptions imply $e(R/I)=e(R/J)-1$. Let $L=J:I$ and note that, by Lemma \ref{ses}, ${\rm pd}(R/I)\leq {\rm pd}(R/L)+1$. Moreover, $R/L$ is unmixed with $e(R/L)=e(R/J)-e(R/I)=1$ (by Proposition \ref{linkage}), therefore, by a well-known result of Samuel (see Proposition \ref{almostCM}), $R/L$ is Cohen-Macaulay. Then ${\rm pd}(R/I)\leq g+1$. Hence, by the Auslander-Buchsbaum formula, we have ${\rm depth}(R/I)\geq {\rm dim}(R/I)-1$. Finally, by assertion (b), $R/I$ is not Cohen-Macaulay, because ${\rm deg}(F)> s(R/J)$. This yields ${\rm depth}(R/I)={\rm dim}(R/I)-1$.
\end{proof}

The next simple example shows that the bound given in Theorem \ref{bound} can be (trivially) sharp, and there are ideals of maximal multiplicity that are not Cohen-Macaulay. 

\begin{ex}\label{nonCM}
Let $R=k[x,y,z]$, $J=(x^2,xy,y^2)$ and $F=xz$. Then $I=J+(F)$ is not unmixed, $R/I$ has maximal multiplicity and is almost Cohen-Macaulay.
\end{ex}

We now apply Theorem \ref{Main} to almost complete intersection ideals generated by quadrics. 
\begin{cor}\label{onedeg}
Let $I$ be an almost complete intersection of height $g\geq 1$ generated by homogeneous elements of the same degree $2$. 
If $e(R/I)=2^g-(g-1)$, then $R/I$ has maximal multiplicity and is Cohen-Macaulay.
\end{cor}

\begin{proof}
It follows by Corollary \ref{En2} together with Theorem \ref{Main} (b).
\end{proof}

If $I$ is an almost complete intersection generated by $4$ quadrics (that is, $g=3$ and $d=2$), then  Corollary \ref{En2} gives $e(R/I)\leq 6$, and Corollary \ref{onedeg} yields that $R/I$ is Cohen-Macaulay if $e(R/I)=6$. This fact is employed in \cite{HMMS2}. 

The following remark shows that, without further assumptions, no other values of $e(R/I)$ guarantee the Cohen-Macaulayness of $R/I$.
\begin{rmk}\label{4quad}
Let $I$ be a height three ideal generated by four quadrics. 
\begin{itemize}
\item[(a)] If $e(R/I)=6$, then $R/I$ is Cohen-Macaulay$;$
\item[(b)] For every $1\leq e \leq 6$, there are examples of $I$ with $e(R/I)=e$ and, if $e\neq 6$, $R/I$ is not Cohen-Macaulay.
\end{itemize} 
\end{rmk}
Assertion (a) follows by Corollary \ref{onedeg}. Assertion (b) can be seen, for instance, as follows. Take $R=k[a,b,c,x,y,z]$, 
\begin{enumerate}
\item if $I=(ax,by,cz,x^2+y^2+z^2)$, then $e(R/I)=1$ and ${\rm pd}(R/I)=4$;
\item if $I=(ax,by,xy+xz+yz,x^2+y^2+z^2)$, then $e(R/I)=2$ and ${\rm pd}(R/I)=4$;
\item if $I=(ax+by+cz,x^2,y^2,z^2)$, then $e(R/I)=3$ and ${\rm pd}(R/I)=6$;
\item if $I=(ax,x^2,y^2,z^2)$, then $e(R/I)=4$ and ${\rm pd}(R/I)=4$;
\item if $I=(ax+by+cz,bx+cy+az,cx+ay+bz,bx+cy-bz-cz)$, then $e(R/I)=5$ and ${\rm pd}(R/I)=4$;
\item if $I=(x^2,y^2,z^2,xy)$, then $e(R/I)=6$ and, by part (a), $\pd(R/I)={\rm ht}\,I=3$.
\end{enumerate}
\medskip
 
Next, we want to provide a sufficient condition for Cohen-Macaulay ideals to have maximal multiplicity. The first step consists in describing the structure of the colon ideal $J:F$.
\begin{lem}\label{ci}
Let $I$ be a Cohen-Macaulay homogeneous ideal of height $g$ having maximal multiplicity. If $J$ and $F$ form a maximal decomposition of $I$, then $J:F=(x_1,\ldots,x_{g-1},q)$, for some linearly independent linear forms $x_1,\ldots,x_{g-1}$ and an element $q\notin (x_1,\ldots,x_{g-1})$.
\end{lem}
\begin{proof}
From the short exact sequence 
$$(\star)\qquad\qquad 0\longrightarrow R/J:F[-\deg(F)] \longrightarrow R/J \longrightarrow R/I \longrightarrow 0 $$
one obtains $e(R/J)-e(R/I)=e((R/J:F)[-\deg(F)])=e(R/J:F)$.  First, assume $I$ is $\m$-primary. If $\deg(F)\geq s(R/J)$, then, by assumption of maximal multiplicity, $e(R/I)=e(R/J)-1$, whence $e(R/J:F)=1$. Then $J:F=\m$ and the statement follows. We may then assume $\deg(F)<s(R/J)$. Since  $J$ and $F$ form a maximal decomposition of $I$, the proof of Theorem \ref{bound} gives $HF_{R/I}(i)=HF_{R/J}(i)-1$ for all $\deg(F)\leq i \leq s(R/J)$.
This yields that $HF_{R/J:F}(i)=1$ if $0\leq i \leq s(R/J)-\deg(F)$, hence, there exist linearly independent linear forms $x_1,\ldots,x_{g-1},x_g$ such that $J:F=(x_1,\ldots,x_{g-1},x_g^{c})$ and the statement follows.

Now assume ${\rm dim}(R/I)=d>0$. Let $L_1,\ldots,L_d$ be general linear forms in $R$, set $L=(L_1,\ldots,L_d)$ and let $^-$ denote images in $\overline{R}=R/L$. Since $I$ is Cohen-Macaulay, $L_1,\ldots,L_d$ form a regular sequence on $R/I$, then, applying the functor $-\otimes_R R/L$ to the short exact sequence
$(\star)$, one obtains the short exact sequence
$$0\longrightarrow \overline{R}/\overline{J:F} \longrightarrow \overline{R}/\overline{J} \longrightarrow \overline{R}/\overline{I} \longrightarrow 0. $$
Clearly, $\overline{R}=R/L$ is still a polynomial ring, $\overline{I}$ is a homogeneous $\overline{\m}$-primary ideal, $\overline{R}/\overline{I}$ has maximal multiplicity, and $\overline{J}$ and $\overline{F}$ form a maximal decomposition of $\overline{I}$. Then, by the above, one has $\overline{J:F}=\overline{J}:\overline{F}=(\overline{z_1},\ldots,\overline{z_{g-1}},\overline{z_g}^c)$ for some linearly independent linear forms $z_1,\ldots,z_g$ of $\overline{R}$. 
The statement now follows by lifting this equality back to $R$. 
\end{proof}

For the rest of the paper, a complete intersection $C$ of height $g$ containing $g-1$ linearly independent linear forms is called an {\it almost linear complete intersection}.

\begin{prop}\label{ci2}
Let $I$ be a Cohen-Macaulay homogeneous ideal of height $g$ having maximal multiplicity. Then there exists an almost linear complete intersection $C'$ such that $(I+C')/C'$ is cyclic.
\end{prop}
\begin{proof}
If $I$ is contained in an almost linear complete intersection, then $(I+C')/C'=0$ and the statement follows trivially. We may then assume $I$ is not contained in any almost linear complete intersection. Let $J$ and $F$ form a maximal decomposition of $I$. By Lemma \ref{ci}, the ideal $C'=J:F$ is an almost linear complete intersection. Note that $J\subseteq I\cap C'$. Since $I$ is not contained in $C'$, then $(I+C')/C'$ is non-zero. Now, the natural mapping $I/J\rightarrow I/I\cap C'\cong (I+C')/C'\rightarrow 0$ together with the assumption that $I/J$ is cyclic yields that also $(I+C')/C'$ is cyclic.
\end{proof}

Next, we prove a sufficient condition for Cohen-Macaulay ideals to have maximal multiplicity. The assumption on $(I+C')/C'$ being cyclic is  necessary, by Lemma \ref{ci2}.
\begin{thm}\label{char}
Let $I$ be a homogeneous Cohen-Macaulay ideal of height $g$. If there exists an almost linear complete intersection $C'$ satisfying the following two conditions:
\begin{itemize}
\item[$($i$)$] $(I+C')/C'$ is non-zero and cyclic, generated by an element of degree $\delta \geq 1$, and
\item[$($ii$)$] $e(R/C')\leq {\rm max}\{1,s(R/I)-\delta+1\}$,
\end{itemize}
then $I$ has maximal multiplicity.
\end{thm}

\begin{proof}
Let $F$ be a homogeneous element of $I$ of degree $\delta\geq 1$ whose image generates the cyclic module $(I+C')/C'$. Set $J=C'\cap I$ and note that $I=J+(F)$. Set $C=J:F=C':F$. Since $C'\subseteq C$ are unmixed of the same height, the ideal $C$ is again an almost linear complete intersection and $e(R/C)\leq e(R/C')$. Write $C=(x_1,\ldots,x_{g-1},q)$ and note that $e(R/C)={\rm deg}(q)\leq e(R/C')\leq  {\rm max}\{1,s(R/I)-\delta+1\}$. We need to show that $J$ is Cohen-Macaulay and $e(R/I)=e(R/J)-{\rm max}\{1,s(R/J)-\delta+1\}$.

Since $I$ and $C$ are Cohen-Macaulay ideals of height $g$, the short exact sequence 
$$(\star) \qquad\qquad 0\longrightarrow (R/C)[-\delta] \stackrel{\cdot F}{\longrightarrow} R/J \longrightarrow R/I \longrightarrow 0$$
implies that $R/J$ is Cohen-Macaulay too. Now, if $1\geq s(R/I)-\delta+1$, then $e(R/C)\leq 1$ and, since $C$ is a proper ideal, then $e(R/C)=1$, yielding that $e(R/I)=e(R/J)-e(R/J:F)=e(R/J)-1$. This proves that $I$ has maximal multiplicity. We may then assume $s(R/I)-\delta+1> 1$. 
The Horseshoe Lemma applied to $(\star)$ gives $s(R/J)\geq {\rm min}\{s(R/I),s(R/C[-\delta])\}$.
Since $C$ is an almost linear complete intersection, we have
$$s(R/C[-\delta])=\delta+s(R/C)=\delta+(\deg(q)+g-1)-g=\delta+\deg(q)-1,$$ 
whence we obtain
$$s(R/J)\geq {\rm min}\{s(R/I),s(R/C[-\delta])\}={\rm min}\{s(R/I),\delta+\deg(q)+-1\}=\delta+\deg(q)-1,$$
where the last equality holds because the inequalities
$${\rm deg}(q)=e(R/C)\leq e(R/C')\leq {\rm max}\{1,s(R/I)-\delta+1\}=s(R/I)-\delta+1$$
imply that $\delta+\deg(q)-1\leq s(R/I)$.  Then, we have obtained $\delta+\deg(q)-1\leq s(R/J)$, that is, $\deg(q)\leq s(R/J)-\delta+1$. We now apply Theorem \ref{bound} and obtain $$\begin{array}{lll}
e(R/J)-e(R/C) & =\, e(R/I) & \leq \,e(R/J)-{\rm max}\{1,s(R/J)-\delta+1\}\\
              & \leq \,e(R/J)-{\rm max}\{1,{\rm deg}(q)\} & =\, e(R/J)-{\rm deg}(q)\\
              &= \,e(R/J)-e(R/C) &
\end{array}$$
proving that $e(R/I)=e(R/J)-{\rm max}\{1,s(R/J)-\delta+1\}$. 
\end{proof}

Note that the proof of Theorem \ref{char} is constructive, in the sense that if $(I+C')/C'$ is non-zero, cyclic and $e(R/C')\leq  {\rm max}\{1,s(R/I)-\delta+1\}$, then one can explicitly construct a maximal decomposition $I=J+(F)$ of the ideal $I$. 

We isolate the special case where  $C'$ is a {\it linear prime}, that is, where $C'$ is a prime ideal generated by linear forms. 
\begin{cor}\label{char2}
Let $I$ be a homogeneous Cohen-Macaulay ideal of height $g$. If there exists a linear prime $C'$ of height $g$ such that $(I+C')/C'$ is cyclic and non-zero, then $I$ has maximal multiplicity.
\end{cor}
\begin{proof}
By assumption $e(R/C')=1$, so the multiplicity condition of Theorem \ref{char} is trivially satisfied.
\end{proof}

We now exhibit two classes of ideals having maximal multiplicity. 
\begin{ex}\label{m-prim}
Let $I$ be any homogeneous $\m$-primary ideal, then $I$ has maximal multiplicity.
\end{ex}
\begin{proof}
Let $F$ be any minimal generator of $I$, and set $J=I+F\cdot \mm$. Then $I$ has maximal multiplicity because $I=J+(F)$ and $e(R/I)=e(R/J)-1$.
\end{proof}

Also ideals generated by the 2 by 2 minors of catalecticant matrices have maximal multiplicity. For instance, ideals
 defining rational normal curves have maximal multiplicity. 
\begin{cor}\label{matrix}
Fix integers $d\geq 1$, $r\geq 2$ and $N\geq 3$. Then the ideal $I=I_2(A)$ generated by the $2$-minors of the matrix  $$A=\left(\begin{array}{cccccc} 
x_1^d & x_2^d & \ldots & x_{N-r+2}^d\\
x_2^d & x_3^d & \ldots & x_{N-r+1}^d\\
\vdots & \vdots & \vdots & \vdots\\
x_r^d & x_{r+1}^d & \ldots & x_{N+1}^d
\end{array} \right) $$
has maximal multiplicity.

In particular, the defining ideals of the rational normal curves of $\mathbb P^N$ for any $N\geq 3$ have maximal multiplicity.
\end{cor}
\begin{proof}
It is known that ${\rm ht}(I)=N-1$. Set $C'=(x_2,\ldots,x_{N})$, and note that $(I+C')/C'$ is cyclic and non-zero, generated by the image of $F=x_1^dx_{N+1}^d-x_r^dx_{N-r+2}^d$ in $R/C'$. Then, by Corollary \ref{char2}, $I$ has maximal multiplicity, and a maximal decomposition of $I$ is given by $I=J+(F)$, where $J$ is the ideal generated by all the 2 by 2 minors of $A$ except for $F$. Rational normal curves consist of the special case where $r=2$ and $d=1$.
\end{proof}
We remark that, in contrast with the ideals $I$ satisfying equality in Corollary \ref{En2}, these ideals can be generated in arbitrarily high degrees and can have arbitrarily large heights.

We conclude this section with a class of almost complete intersections $I$ of height $3$, generated in a single degree $d\geq 4$ and having maximal multiplicity (compare with the discussion after Corollary \ref{En2}). Note that, for these ideals, there is no linear prime $C'$ of height 3 such that $(I+C')/C'$ is cyclic, hence one cannot use Corollary \ref{char2} to prove that $I$ has maximal multiplicity.
\begin{ex}\label{aci}
Fix $t\geq 1$, let $f_i$, $g_i$, $h_i$, where $i=1,2$, be irreducible polynomials of
the same degree $t+1$ in disjoint sets of variables such that $f_1$ is
contained in a linear prime of height $2$. Set $F=h_1h_2$ and $J=(f_1f_2, g_1g_2, f_1g_1h_1)$. Then 
$I =J+(F)$ is a Cohen-Macaulay almost complete
intersection that has maximal multiplicity. 
\end{ex}

\begin{proof}
We first show that $I$ is Cohen-Macaulay. Observe that the ideals $(f_l,g_i,h_j)$ with $1\leq l\leq 2$, $1\leq l=i\leq 2$ and $1\leq j\leq 2$ are all prime. Since
$$I=\left(\bigcap_{1\leq i \leq 2, 1\leq j \leq 2}(f_1,g_i,h_j)\right) \cap (f_2,g_1,h_1)\cap (f_2,g_1,h_2)\cap (f_2,g_2,h_1),$$
then $I$ is unmixed. Also, since $C_0=(f_1f_2,g_1g_2,h_1h_2)$ is a complete intersection of height 3 contained in $I$, then $I\sim C_0:I$ by Proposition \ref{PS}. Since $C_0:I=(f_2,g_2,h_2)$ is a complete intersection, then, by Proposition \ref{linkage}, the ideal $I$ is Cohen-Macaulay.

Let $(x_1,x_2)$ be a linear prime of height $2$ containing $f_1$, and set $C'=(x_1,x_2,g_1)$. Then $(I+C')/C'$ is cyclic, generated by the image of $F=h_1h_2$ in $R/C'$. We have
$s(R/I)=5t+2$, hence $s(R/I)-\deg(F)+1=(5t+2)-(2t+2)+1=3t+1>3$ for every $t\geq 1$. Then
$$e(R/C')=2< 3t+1={\rm max}\{1,s(R/I)-\deg(F)+1\}.$$
Then $I$ has maximal multiplicity by Theorem \ref{char}.
\end{proof}
For instance, let $R=k[x_1,\ldots,x_8,y_1,\ldots,y_8,z_1,\ldots,z_8]$,  
$f_1=x_1^tx_2-x_3^tx_4$, $g_1=y_1^ty_2-y_3^ty_4$, $h_1=z_1^tz_2-z_3^tz_4$, 
$f_2=x_5^tx_6-x_7^tx_8$, $g_2=y_5^ty_6-y_7^ty_8$, $h_2=z_5^tz_6-z_7^tz_8$, and $F=h_1h_2$.
Then $I=(f_1f_2,g_1g_2,f_1g_1h_1,F)$ is a Cohen-Macaulay almost complete intersection that has maximal multiplicity and is generated in degree $2t+2\geq 4$ for any $t\geq 1$.

\section{Application to Quasi-Gorenstein rings}

In this section we prove a lower bound for the multiplicity of quasi-Gorenstein rings and a sufficient condition for a quasi-Gorenstein ring to be Gorenstein. 
\bigskip

We first recall the definition of graded quasi-Gorenstein rings (also known as $1$-Gorenstein rings). Recall that the canonical module of a $d$-dimensional graded ring $S$ with homogeneous maximal ideal $\mm_S$ is defined as $\omega_S={\rm Hom}_k(H_{\mm_S}^d(S),E)$, where $E=E_S(S/\mm_S)$ is the injective envelope of the residue field of $S$.

A graded ring $S$ is {\em quasi-Gorenstein} if $\omega_S\cong S(a)$ for some integer $a$. An ideal $Q$ is said to be {\em quasi-Gorenstein} if $R/Q$ is quasi-Gorenstein. The number $a=a(R/Q)$ is the $a${\it -invariant} of $R/Q$. Note that quasi-Gorenstein ideals are unmixed (indeed, their factor rings satisfy Serre's property $(S_2)$).

Quasi-Gorenstein rings arise naturally in several contexts, for instance, (extended) Rees algebras (cf. \cite[Theorem~2.8]{VZ}, \cite[Theorem~3.2]{JU}, and \cite[Theorem~6.1]{HKU}) or coordinate rings of cones over abelian surfaces (see, for instance, \cite{Sc}). From the definition, it follows that an ideal $Q$ of $R$ is Gorenstein if and only if $Q$ is quasi-Gorenstein and Cohen-Macaulay. We will employ the following result, essentially proved by Schenzel \cite[Proposition~1]{Sc2}.
\begin{prop}\label{linkQG}
Let $Q$ be an $R$-ideal. The following are equivalent:
\begin{itemize}
\item[(i)] $Q$ is quasi-Gorenstein;
\item[(ii)] there exist a Gorenstein ideal $G\subseteq Q$ and $h\in R$ so that $Q\sim G+hR$ by $G$.
\end{itemize}  
\end{prop}

We now give a lower bound for the multiplicity of graded quasi-Gorenstein rings.
\begin{prop}\label{boundQ}
If $Q$ is a homogeneous quasi-Gorenstein ideal, then 
$$e(R/Q)\geq {\rm max}\left\{1, a(R/Q)+{\rm dim}(R/Q)+1\right\}.$$
\end{prop}
\begin{proof}
We may assume $Q$ is proper.
 By Proposition \ref{linkQG}, there exists a Gorenstein ideal $G\subseteq Q$ with ${\rm ht}\,G={\rm ht}\,Q$ and $Q\sim I$,
 where $I=G+hR$ for some element $h$. Note that $h\notin G$ (otherwise $I=G$ and then, by Proposition \ref{PS}, $Q=G:G=R$, contradicting $Q$ is proper) and ${\rm ht}\,I={\rm ht}\,G$. Then, by Proposition \ref{linkage} (b) and Theorem \ref{bound}, we have
$$e(R/Q)=e(R/G)-e(R/I)\geq {\rm max}\{1, s(R/G)-{\rm deg}(h)+1 \}.$$
Now, from the standard exact sequence from linkage
$$0\rightarrow G \rightarrow I \rightarrow \omega_{R/Q}(-a(R/G)) \rightarrow 0$$
we have $\omega_{R/Q}\cong I/G[a(R/G)]$. Since $I/G$ is generated by the image of $h$, we obtain $a(R/Q)=a(R/G)-{\rm deg}(h)$. This fact, together with the equalities $s(R/G)={\rm dim}(R/G)+a(R/G)$ and ${\rm dim}(R/G)={\rm dim}(R/Q)$, yields
$$\begin{array}{ll}
e(R/Q)&\geq {\rm max}\{1,s(R/G)-\deg(h)+1\}\\
      &={\rm max}\{1, {\rm dim}(R/G)+a(R/G)-\deg(h)+1 \}\\
      &={\rm max}\{1, \dim(R/Q)+a(R/Q)+1 \}.
      \end{array}$$
\end{proof}

We now prove a multiplicity-based sufficient condition for $R/Q$ to be Gorenstein.
\begin{thm}\label{Gor}
Let $Q$ be a homogeneous quasi-Gorenstein ideal. If 
$$e(R/Q)= {\rm max}\left\{1, a(R/Q)+{\rm dim}(R/Q)+1\right\},$$
then $R/Q$ is Gorenstein.
\end{thm}

\begin{proof}
We may assume $Q$ is proper. Let $G\subseteq Q$ be a Gorenstein ideal with ${\rm ht}\,G={\rm ht}\,Q$, and set $I=G:Q$.
From the proof of Proposition \ref{boundQ}, the given equality implies that $I$ has maximal multiplicity. 

Since ${\rm deg}(h)=a(R/G)-a(R/Q)$ and $-a(R/Q)\leq {\rm dim}(R/Q)$, then we have ${\rm deg}(h)\leq s(R/G)$. By Theorem \ref{Main} (b), this yields that $R/I$ is Cohen-Macaulay. Thus, by Proposition \ref{linkage} (a), $R/Q$ is Cohen-Macaulay and, then, Gorenstein.
\end{proof}

Next, we would like to point out an analogy between the two extremal values of $e(R/I)$. 
Assume $I=J+(F)$, where $J$ is Gorenstein and $F\notin J$ is a zero divisor on $R/J$, then
$$1\leq e(R/I)\leq e(R/J)-{\rm max}\{1, s(R/J)-{\rm deg}(F) +1\}.$$
If $I$ has maximal multiplicity, by Theorem \ref{Main} either $R/I$ is Cohen-Macaulay or is almost Cohen-Macaulay. When $e(R/I)=1$ a similar statement holds.

\begin{prop}\label{almostCM}
Let $J$ be a homogeneous Gorenstein ideal, and let $F\notin J$ be a homogeneous element such that ${\rm ht}\,I={\rm ht}\,J$, where $I=J+(F)$. Assume $e(R/I)=1$.\begin{itemize}
\item[(a)]$(${\rm Samuel} \cite{Sa}, \cite{Na}$)$ $I$ is unmixed if and only if $I$ is Cohen-Macaulay.
\item[(b)] $I$ is not unmixed if and only if ${\rm depth}(R/I)={\rm dim}(R/I)-1$ $($that is, $R/I$ is almost Cohen-Macaulay$)$. 
\end{itemize} 
\end{prop}

\begin{proof}
We only prove assertion (b). If ${\rm depth}(R/I)={\rm dim}(R/I)-1$, then $R/I$ is not Cohen-Macaulay, hence, by assertion (a), $I$ is not unmixed. Next, assume $I$ is not unmixed. Set $L=J:I$ and note that $L\sim I^{un}$, by Proposition \ref{PS}. Since $e(R/I^{un})=e(R/I)=1$, the ideal $I^{un}$ is Cohen-Macaulay, by assertion (a). By Proposition \ref{linkage} (a), $R/L$ is also Cohen-Macaulay. Then, Lemma \ref{ses} gives ${\rm pd}(R/I)\leq {\rm grade}\,I+1$. An application of the Auslander-Buchsbaum formula now proves ${\rm depth}(R/I)\geq {\rm dim}(R/I)-1$. Since $I$ is not unmixed, we have ${\rm depth}(R/I)\neq{\rm dim}(R/I)$, yielding ${\rm depth}(R/I)={\rm dim}(R/I)-1$.
\end{proof}

It is then natural to ask whether the analogue of Proposition \ref{almostCM} holds: assume $I$ has maximal multiplicity, is it true that $I$ is unmixed if and only if $I$ is Cohen-Macaulay? Corollary \ref{Unm} provides a positive answer.

\begin{cor}\label{Unm}
Let $J$ be a homogeneous Gorenstein ideal, and let $F\notin J$ be a homogeneous element such that ${\rm ht}\,I={\rm ht}\,J$, where $I=J+(F)$. Assume $I$ has maximal multiplicity.
\begin{itemize}
\item[(a)] $I$ is unmixed if and only if $I$ is Cohen-Macaulay if and only if ${\rm deg}(F)\leq s(R/J)$.
\item[(b)] $I$ is not unmixed if and only if ${\rm depth}(R/I)={\rm dim}(R/I)-1$ 
if and only if ${\rm deg}(F)>s(R/J)$. 
\end{itemize} 
\end{cor}

\begin{proof} 
 (a) By Theorem \ref{Main} (c) we only need to prove that, if $I$ is unmixed, then $I$ is Cohen-Macaulay. If $I$ is unmixed, the ideal $Q=J:I$ is quasi-Gorenstein by Proposition \ref{linkQG}. As in the proof of Proposition \ref{boundQ}, the maximal multiplicity of $I$ yields $e(R/Q)= {\rm max}\left\{1, a(R/Q)+{\rm dim}(R/Q)+1\right\}$. Now, Theorem \ref{Gor} implies that $Q$ is Gorenstein and then, by Proposition \ref{linkage} (a), $I$ is Cohen-Macaulay.

Assertion (b) follows from assertion (a) and Theorem \ref{Main} (a) and (c). 
\end{proof}

We conclude with  a question. If $I$ is an ideal of maximal multiplicity, in general there could be several different maximal decompositions $I=J+(F)$. 
Moreover, if $F$ and $F'$ are minimal generators of $I$, there could be a maximal decomposition of $I$ of the form $I=J+(F)$, but no maximal decomposition for $I$ of the form $I=J'+(F')$. These observations raise the following question: {\it Is there an implicit characterization of ideals $I$ of maximal multiplicity?} Theorem \ref{char} gives a sufficient condition when $I$ is Cohen-Macaulay, and Proposition \ref{ci2} gives a necessary condition, but we don't know a necessary and sufficient condition..

\section{Acknowledgements}
Part of these results were proved while all the authors were members of the MSRI program on Commutative Algebra. We warmly thank the MSRI for its support and the great mathematical environment. We are also grateful to G. Caviglia, L. Fouli, M. Kummini and B. Ulrich for helpful conversations.


\end{document}